\documentclass[12pt,a4paper,reqno]{amsart}
\usepackage{graphicx}
\usepackage[linktocpage=true,colorlinks,citecolor=blue,linkcolor=blue,urlcolor=blue]{hyperref}
\usepackage{amssymb}
\usepackage{cite}
\usepackage{bbm}
\usepackage{amsmath}
\usepackage{latexsym}
\usepackage{amscd}
\usepackage{tikz}
\usepackage{amsthm}
\usepackage{mathrsfs}
\usepackage{url}
\usepackage[utf8]{inputenc}
\usepackage[english]{babel}
\usepackage{amsfonts}
\usepackage{mathtools}
\usepackage{comment}
\usepackage[autostyle]{csquotes}
\usepackage[colorinlistoftodos]{todonotes}

\numberwithin{equation}{section}
\def\R{\mathbb R}
\def\Z{\mathbb Z}
\def\C{\mathbb C}

\def\N{\mathbb N}
\def\E{\mathbb E}

\def\ee{\varepsilon}

\def\gcd{\operatorname{gcd}}

\DeclarePairedDelimiter\ceil{\lceil}{\rceil}
\DeclarePairedDelimiter\floor{\lfloor}{\rfloor}

\newtheorem{theorem}{Theorem}[section]
\newtheorem{lemma}[theorem]{Lemma}
\newtheorem{proposition}[theorem]{Proposition}

\newtheorem{corollary}[theorem]{Corollary}

\theoremstyle{remark}
\newtheorem{remark}[theorem]{Remark}

\theoremstyle{definition}

\theoremstyle{remark}

\numberwithin{equation}{section}

\begin{document}
\title{Paucity phenomena for polynomial products}
% \todo[inline]{feel free to suggest new titles}
\author{Victor Y. Wang}
\address{Courant Institute, 251 Mercer Street, New York 10012, USA}
\address{IST Austria, Am Campus 1, 3400 Klosterneuburg, Austria}
\email{vywang@alum.mit.edu}
\author{Max Wenqiang Xu}
\address{Department of Mathematics, Stanford University, Stanford, CA, USA}
\email{maxxu@stanford.edu}
\subjclass{Primary 11K65; Secondary 11D45, 11D57, 11D79, 11N37}
\keywords{Multiplicative Diophantine equations, paucity, random multiplicative functions, Gaussian behavior}
% unlikely divisors

% \date{\today}

\begin{abstract}
Let $P(x)\in \mathbb{Z}[x]$ be a polynomial with at least two distinct complex roots. We prove that the number of solutions $(x_1, \dots, x_k, y_1, \dots, y_k)\in [N]^{2k}$ to the equation
\[
    \prod_{1\le i \le k} P(x_i) = \prod_{1\le j \le k} P(y_j)\neq 0
\]
(for any $k\ge 1$)
is asymptotically $k!N^{k}$ as $N\to +\infty$.
This solves a question first proposed and studied by Najnudel.
The result can also be interpreted as saying that  all even moments of random partial sums $\frac{1}{\sqrt{N}}\sum_{n\le N}f(P(n))$ match standard complex Gaussian moments as $N\to +\infty$, where $f$ is the Steinhaus random multiplicative function.
\end{abstract}

\maketitle
\section{Introduction}

We are interested in counting solutions to Diophantine equations involving products of polynomial values.
Let $P(x) \in \Z[x]$ be a polynomial.
Consider the equation
\begin{equation}\label{eqn: S}
    \prod_{1\le i \le k} P(x_i) = \prod_{1\le j \le k} P(y_j)\neq 0
\end{equation}
for a given integer $k\ge 1$.
For integers $N\ge 1$, let $[N]:=\{1,2,\dots, N\}$ and define
\[A_{P, 2k}([N]) := \#\{(x_1, x_2, \dots, x_k, y_1, y_2, \dots, y_k)\in [N]^{2k}: \text{solutions to \eqref{eqn: S}} \}.\]
Call a solution to \eqref{eqn: S} \emph{trivial} if the first $k$ variables $x_1,\dots,x_k$ equal the last $k$ variables $y_1,\dots,y_k$ in pairs.
We say \eqref{eqn: S} has a \emph{paucity} of nontrivial solutions in $[N]^{2k}$ as $N\to +\infty$
if almost all solutions are trivial,
or equivalently, if $\lim_{N\to +\infty} (A_{P,2k}([N])/N^k) = k!$ holds.
One might expect a typical equation of the form \eqref{eqn: S} to have this paucity property,
which is a problem raised and studied in some special cases of $P(x)$ by Najnudel in \cite[\S5]{Naj}.
A recent result of Klurman, Shkredov, and Xu \cite[Proposition~1.4]{KSX} confirms that such phenomena exist in the $k=2$ case.

\subsection{Main results}

In this paper, we prove the paucity phenomenon occurs for all $P(x)$ satisfying mild necessary conditions. 
Let $P(x)\in \Z[x]$ have degree 
$d\ge 2$ and at least two distinct complex roots.
Let $e_P$ be the maximum multiplicity of any complex root of $P$; then $1\le e_P\le d-1$.
(For example, if $P(x) = x^2(x+1)$, then $e_P = 2$.)

%Assume that $P(n) > 0$ for all positive integers $n$.

\begin{theorem}
\label{THM:polynomial-paucity}
Let $P, d, e_P$ be as above.
Then for integers $k, N\ge 1$, we have
\[A_{P, 2k}([N])  =  k! N^{k} + O_{ P, k, \ee}(N^{k-1/6e_P+\ee}).  \]
% In particular, \eqref{eqn: S} has a paucity of nontrivial solutions in $[N]^{2k}$.
\end{theorem}

We prove Theorem~\ref{THM:polynomial-paucity} in \S\ref{SEC:proof-of-main-result}.
Related paucity results (for linear $P(x) = x-\alpha$ with $\alpha \in \mathbb{C}$ \emph{irrational}) were obtained independently by Bourgain, Garaev, Konyagin, and Shparlinski \cite[Corollary~27]{Bourgain2014} and Heap, Sahay, and Wooley \cite[Theorems~1.1 and~1.2]{HEAP2022}.
By considering certain algebraic norms, it may be possible to use our result (Theorem~\ref{THM:polynomial-paucity}) to obtain a version of those results (of \cite{Bourgain2014} and \cite{HEAP2022}), weaker in the error term.
However, we are not aware of any known or easy implication in the other direction.

\subsection{Linear case and optimality of the results}

The set of polynomials $P$ we allow in Theorem~\ref{THM:polynomial-paucity} is precisely $\Z[x]\setminus \{c(ax-r)^m: c,a,r\in \Z,\; m \in \mathbb{N} \}$.
Since \eqref{eqn: S} is ``homogeneous'' in $P$,
the condition $P(x)\neq c(ax-r)^m$ in Theorem~\ref{THM:polynomial-paucity} basically requires $P(x)$ to not be linear.
There has been a lot of literature in the case $P(x)=x$, e.g.~\cite{batyrev1995manin, Breteche, breteche2001toric, GS2001, Harperhigh, HarperHelson, HeapLind};
it is known that in this case, $A_{P,2k}([N])$ is larger than the number of trivial solutions by a power of $\log N$ with degree depending on $k$.
Thus, there is no paucity phenomenon in the linear case and our mild condition is necessary. 

However, we remark that in the companion paper \cite{WX2022} of the authors and Pandey, the analogous counting problem is solved for $P(x)=x$ with short interval support $[N, N+H]$, $H\le N/(\log N)^{Ck^2}$,\footnote{Previously, Bourgain, Garaev, Konyagin, and Shparlinski \cite[proof of Theorem~34]{Bourgain2014} handled a range of the form $H\le N/\exp(C_k\log{N}/\log\log{N})$.}
and this is one of the key steps in making progress towards a recent question of Harper \cite{Harper-moving}.
The \enquote{paucity} in this case comes from the \enquote{shortness} of the interval;
in comparison, the \enquote{paucity} in the present paper comes from the fact that $P(x)$ is ``genuinely of degree $\ge 2$'' (i.e.~it has at least two distinct complex roots).

\subsection{Application to random multiplicative functions}

% (Not the only origin: de la Breteche for example is another origin.)
One motivation for the study of Diophantine equations of the form \eqref{eqn: S} (when $P(x)=x$) comes from the moments of the Steinhaus random multiplicative function:
let $f(p)$ be independent random variables (for all primes $p$) taking values uniformly from the complex unit circle, and define $f(n)$ completely multiplicatively for all $n$.
We can use orthogonality to derive that
\[\E_f \Bigl|\sum_{n\le N}f(n)\Bigr|^{2k} = A_{P(n)=n, 2k}([N]). \]
Computing the moments here is exactly the same as counting solutions to \eqref{eqn: S}. 
As mentioned before, we have now a good understanding of the high moments of $\frac{1}{\sqrt{N}}\sum_{n\le N}f(n)$.
When the support of $f(n)$ is the full set  $[N]$, these moments blow up and grow like powers of $\log N $ as $k$ grows.
Our recent work  \cite{WX2022} with Pandey proves that when $f(n)$ is supported on short intervals $[N, N+H]$ with $H\to +\infty$ and $H\ll N/(\log N)^{Ck^{2}}$ as $N \to +\infty$, these moments are consistent with Gaussian moments for all $k\ge 1$.

Theorem~\ref{THM:polynomial-paucity} can be interpreted as the following moment asymptotic for the Steinhaus random multiplicative function $f$.
Suppose that $P(n)$ has a positive leading coefficient;
then there exists a constant $n_0\ge 0$ such that $P(n)>0$ for all $n>n_0$.
Then
\[\E_f \Bigl| \frac{1}{\sqrt{N}}\sum_{n_0<n\le N} f(P(n))\Bigr|^{2k} = k! + O_{k, P, \ee}(N^{-1/6e_P+\ee}).\]
The moments of random sums are strongly connected to their limiting distribution.
In particular, to prove the random sums have Gaussian limiting distribution, it is sufficient to show that all the moments match the  Gaussian moments (although this is not a necessary condition).
Thus, Theorem~\ref{THM:polynomial-paucity} potentially gives a way to establish the central limit theorem for 
$\frac{1}{\sqrt{N}}\sum_{n\le N} f(P(n))$ subject to establishing the odd moments,\footnote{Very recently, Besfort Shala noted that the odd moments are trivial to handle by the divisor bound.}
which could likely be done either by the methods of the present paper, or by adapting the treatment of analogous odd moments in \cite[end of \S3]{WX2022}.
In fact, Klurman, Shkredov, and Xu \cite{KSX} already proved that 
$\frac{1}{\sqrt{N}}\sum_{n\le N} f(P(n))$ has Gaussian limiting distribution by exploiting the martingale structure, which successfully reduces the task of establishing a central limit theorem to a fourth moment estimate. There is a lot more literature on the study of the distribution when $P(n)$ is linear and we refer readers to \cite{CS, SoundXu, Harper, Hough}.

\subsection{Notation}
Our notation is standard.
For any two functions $f, g: \R\to \R$, we write $f\ll g, g\gg f, g= \Omega(f)$ or $f = O(g)$ if there exists a positive constant $C$ such that $|f|\le C g$.
% , and we write $f \asymp g$ or $f=\Theta(g)$ if $f\ll g$ and $g \gg f$.
We use $O_k$ to indicate that the implied constant depends on $k$.
% We use $o_{X\to +\infty}(g)$ to denote a quantity $f$ such that $f/g$ tends to zero as $X$ tends to infinity. 
Let $\omega(n)$ be the number of distinct prime factors of $n$, and $\tau_k(n)$ be the number of ways to write $n$ as an ordered product of $k$ positive integers.
We use $\lfloor x \rfloor$ to denote the largest integer $\le x$.

\subsection*{Acknowledgement}
We thank Oleksiy Klurman, Ilya Shkredov, and Igor Shparlinski for helpful comments on earlier versions of the paper,
and thank Yotam Hendel for providing a reference for Lemma~\ref{LEM:univariate-polynomial-congruence-bound}.
We also thank the anonymous referee for their generous corrections and comments. 
The first author has received funding from the European Union's Horizon~2020 research and innovation program under the Marie Sk\l{}odowska-Curie Grant Agreement No.~101034413. The second author is partially supported by the Cuthbert C. Hurd Graduate Fellowship in the Mathematical Sciences, Stanford.

\section{Proof of the main result}
\label{SEC:proof-of-main-result}

The proof of Theorem~\ref{THM:polynomial-paucity} shares some ideas with the authors' recent work \cite{WX2022} with Pandey, particularly the use of ``factored'' congruence conditions and induction on $k$.\footnote{The work \cite[proof of Lemma~22 (for Lemma~23 and Theorem~34)]{Bourgain2014} also uses congruence conditions,
but does not ``factor'' them as we do in Corollary~\ref{COR:unlikely-P-divisor} below (and as in \cite[\S2 and \S3]{WX2022}).}
% \todo[inline]{want to briefly mention some common idea, maybe one sentence or something; then mention the additional ideas}
But the present proof also requires a novel use of
some algebra (e.g.~a result of Huxley \cite{HuxleyCongruence})
and algebraic geometry (e.g.~\cite[Theorem~4]{BP} of Bombieri and Pila),
among other ingredients.
More specifically, our argument requires us to separately consider three types of solutions: trivial solutions (for which $x_i$ and $y_j$ are equal in pairs), ``almost-trivial'' nontrivial solutions (for which some $\gcd(P(x_i), P(y_j))$ is very large), and ``generic'' nontrivial solutions. 
For the two types of nontrivial solutions, see especially Corollary~\ref{COR:unlikely-P-divisor} and its proof. Roughly speaking,
we use \cite{HuxleyCongruence} to show that there are relatively few ``generic'' nontrivial solutions,
while we use \cite[Theorem~4]{BP} to show that \emph{on average} there are relatively few ``almost-trivial'' nontrivial solutions.

We now build up to a proof of Theorem~\ref{THM:polynomial-paucity}.
Let $P, d, e_P$ be as in Theorem~\ref{THM:polynomial-paucity}.
Furthermore, assume that $P$ has positive leading coefficient.
(Theorem~\ref{THM:polynomial-paucity} is invariant under scaling $P$, so we may certainly make this assumption.)
Now fix an integer $n_0\ge 0$ such that $P(n)>0$ for all $n>n_0$.
By the divisor bound, \eqref{eqn: S} has at most $O_\ee(N^{k-1+\ee})$ solutions in $[N]^{2k}$ with $\min(x_1, \dots, x_k, y_1, \dots, y_k) \le n_0$.
Therefore, we may and do restrict attention to $n>n_0$.
By replacing $P(x)$ with $P(x+n_0)$, we may then assume $P(n)>0$ for all $n>0$, as we do from now on.

Let $Q(x)\in \Z[x]$ be a polynomial, with no repeated complex roots, such that
\begin{equation}
\label{COND:Q|P|Q^e_P}
    Q(x) \mid P(x) \mid Q(x)^{e_P}.
\end{equation}
Let $\Delta_Q\neq 0$ be the discriminant of $Q$.
(On a first reading, we suggest restricting to the special case $P(x) = x(x+1)$ considered by Najnudel in \cite[\S5]{Naj}, where $e_P = 1$ and one can take $Q(x) = x(x+1)$ with $\Delta_Q = 1$.)

The following result on polynomial congruences is due to Huxley \cite{HuxleyCongruence}.
See also the stronger result in Stewart \cite[Corollary~2]{StewartCongruence}, and the ensuing historical discussion there.

\begin{lemma}
[Huxley]
\label{LEM:univariate-polynomial-congruence-bound}
Let $P, d, e_P, Q, \Delta_Q$ be as above.
Let $\ell\ge 1$ be an integer.
Then
\begin{equation}
\label{INEQ:Huxley-bound}
\#\{x\bmod{\ell}: Q(x)\equiv 0\bmod{\ell}\}
% \le d^{\omega(\ell)} \gcd(\ell,\Delta_Q). \]
\le d^{\omega(\ell)} \lvert \Delta_Q \rvert^{1/2}.
\end{equation}
\end{lemma}

\begin{proof}
Since $P$ has $\ge 2$ distinct complex roots, the condition \eqref{COND:Q|P|Q^e_P} implies $\deg{Q}\ge 2$.
Since $\deg{Q}\ge 2$, the bound \eqref{INEQ:Huxley-bound} then follows directly from \cite{HuxleyCongruence}.
% Note: The bound \eqref{INEQ:Huxley-bound} fails when $\deg{Q}=1$, unless the content of $Q$ is coprime to $\ell$.
\end{proof}

This result leads to the following proposition:
\begin{proposition}
\label{PROP:fixed-polynomial-divisibility}
Let $P, d, e_P, Q, \Delta_Q$ be as above.
Let $z, N\ge 1$ be integers.
Then 
\begin{equation}\label{eqn: mod l}
\#\{x \in [N]: z\mid P(x)\}
\le d^{\omega(z)} \lvert \Delta_Q \rvert^{1/2}
\cdot (1 + N/z^{1/e_P}).
\end{equation}
\end{proposition}

\begin{proof}
Let $\ell$ be the smallest positive integer such that $z \mid \ell^{e_P}$.
Clearly $\ell\mid z$ and $\ell\ge z^{1/e_P}$.

If $x\in [N]$ and $z \mid P(x)$, then $z \mid Q(x)^{e_P}$ by \eqref{COND:Q|P|Q^e_P}, so $\ell \mid Q(x)$.
This implies that
\begin{equation}
\label{INEQ:P-mod-z-to-Q-mod-ell}
\#\{x \in [N]: z \mid P(x)\} \le \#\{x \in [N]: \ell \mid Q(x)\}.
\end{equation}
But a short calculation using Lemma~\ref{LEM:univariate-polynomial-congruence-bound} gives the bound
\begin{equation}
\label{INEQ:Q-mod-ell-incomplete-bound-via-complete-bound}
\#\{x \in [N]: \ell \mid Q(x)\}
\le d^{\omega(\ell)} \lvert \Delta_Q \rvert^{1/2}
\cdot \ceil{N/\ell}
\le d^{\omega(\ell)} \lvert \Delta_Q \rvert^{1/2}
\cdot (1 + N/\ell).
\end{equation}
(For $N=\ell$, the bound \eqref{INEQ:Q-mod-ell-incomplete-bound-via-complete-bound} follows directly from Lemma~\ref{LEM:univariate-polynomial-congruence-bound};
for $N<\ell$ it then follows by enlarging $N$ to $\ell$,
while for $N>\ell$ it follows by reduction modulo $\ell$.)

Since $\ell\mid z$ and $\ell\ge z^{1/e_P}$, the proposition follows from \eqref{INEQ:P-mod-z-to-Q-mod-ell} and \eqref{INEQ:Q-mod-ell-incomplete-bound-via-complete-bound}.
\end{proof}
\begin{remark}
By using \cite[Theorem 1]{KS94} instead of ``enlarging $N$ to $\ell$'' (in the case $N<\ell$),
the right-hand side of \eqref{eqn: mod l} could be improved to $d^{\omega(z)}\cdot (1 + \lvert \Delta_Q \rvert^{1/2} N/z^{1/e_P})$, provided the content of $Q$ (the greatest common divisor of the coefficients of $Q$) is relatively prime to $\ell$.
This would save roughly a factor of $\lvert \Delta_Q \rvert^{1/2}$ when $N\asymp z^{1/e_P}/\lvert \Delta_Q \rvert^{1/2}$. 
But \eqref{eqn: mod l} is sufficient for our purposes (and does not require the content of $Q$ to be coprime to $\ell$). 
\end{remark}

We next deduce the following general counting result.
Here and later, we let
\begin{equation}
\label{EQN:define-lambda-large-gcd-count}
G_{P,\lambda}([N], z) :=
\#\{(x,a,b)\in [N] \times [\lambda]^2:
a\cdot z = b\cdot P(x)\text{ and }a<b\},
\end{equation}
% whenever $z,N,\lambda\in \Z$ with $z,N,\lambda\ge 1$.
whenever $z,N,\lambda\in \N$.
% (We allow $z\le 0$ just to keep Lemma~\ref{LEM:BP-bound} clean.)

% \todo[inline]{V 10/26/23: There seems to be a very minor issue in the endgame with negative values of $P(x)$; to fix it (after we hear back), we can just put absolute values everywhere in \eqref{EQN:define-lambda-large-gcd-count}, \eqref{quantity-T}, and Lemmas \ref{LEM:no-linear-factors} and \ref{LEM:BP-bound}.}

\begin{corollary}
\label{COR:unlikely-P-divisor}
Let $P, d, e_P, Q, \Delta_Q$ be as above.
Let $k, z, N\ge 1$ be integers.
Let
\begin{equation}
\label{quantity-T}
\begin{split}
T: = \#\{(x_1,x_2,\dots,x_k)\in [N]^k:{}
&z\mid P(x_1)P(x_2)\cdots P(x_k), \\
&P(x_1),P(x_2),\dots,P(x_k) < z\}.
\end{split}
\end{equation}
Then for any integer $\lambda\ge 1$, we have
\begin{equation}\label{eqn: 2.6}
T \le
k\cdot G_{P,\lambda}([N], z)\cdot N^{k-1}
+ \tau_k(z)\cdot O(d^{\omega(z)})^k \lvert \Delta_Q \rvert^{k/2}
\cdot \left(\frac{N^k}{z^{1/e_P}} + \frac{N^{k-1}}{\lambda^{1/e_P}} + N^{k-2}\right).
\end{equation}
\end{corollary}

\begin{proof}
For each $j\in [k]$, let $I_j$ be the contribution to $T$ from tuples $(x_1,\dots,x_k)$ with $\gcd(P(x_j), z) \ge z/\lambda$.
Note that if $x\in [N]$ has the properties $\gcd(P(x), z) \ge z/\lambda$ and $P(x) < z$,
then $\gcd(P(x), z) = z/b$ for some positive integer $b\le \lambda$ dividing $z$, and thus $P(x) = a\cdot z/b$ for some positive integer $a<b$,
since $P(x)>0$.
It follows that for each $j\in [k]$, we have
\begin{equation}
\label{INEQ:isolated-gcd-bound}
I_j\le G_{P,\lambda}([N], z)\cdot N^{k-1}.
\end{equation}

Now let $S$ be the contribution to $T$ from tuples $(x_1,\dots,x_k)$ with $\gcd(P(x_j),z) < z/\lambda$ for all $j\in [k]$.
Given positive integers $u_1,\dots,u_k$ with $u_1u_2\cdots u_k = z$,
let $S_{u_1,\dots,u_k}$ be the contribution to $T$ from tuples $(x_1,\dots,x_k)$ with $\gcd(P(x_j),z) < z/\lambda$ and $u_j\mid P(x_j)$ for all $j\in [k]$.
Then $S_{u_1,\dots,u_k} = 0$ unless $u_1,\dots,u_k < z/\lambda$, in which case we must have $k\ge 2$ and (by Proposition~\ref{PROP:fixed-polynomial-divisibility})
\begin{equation*}
\begin{split}
S_{u_1,\dots,u_k}
&\le O(d^{\omega(z)})^k \lvert \Delta_Q \rvert^{k/2}
\cdot \frac{N^k}{\prod_{j\in [k]} \min(N, u_j^{1/e_P})} \\
&\le O(d^{\omega(z)})^k \lvert \Delta_Q \rvert^{k/2}
\cdot \left(N^{k-2} + \frac{N^{k-1}}{\min_{i\in [k]}(z/u_i)^{1/e_P}} + \frac{N^k}{z^{1/e_P}}\right),
\end{split}
\end{equation*}
% since $z = u_1u_2\cdots u_k$;
% and $z/u_i = \prod_{j\in [k]\setminus \{i\}} u_j$;
because $\prod_{j\in [k]} \min(N, u_j^{1/e_P})$ is either $\ge N^2$,
or equal to $N\prod_{j\in [k]\setminus \{i\}} u_j^{1/e_P} = N (z/u_i)^{1/e_P}$,
or else equal to $\prod_{j\in [k]} u_j^{1/e_P} = z^{1/e_P}$.
Therefore,
\begin{equation}
S_{u_1,\dots,u_k}
\le O(d^{\omega(z)})^k \lvert \Delta_Q \rvert^{k/2}
\cdot \left(N^{k-2} + \frac{N^{k-1}}{\lambda^{1/e_P}} + \frac{N^k}{z^{1/e_P}}\right),
\end{equation}
since we may assume $u_1,\dots,u_k < z/\lambda$.
As $S\le \sum_{u_1u_2\cdots u_k = z} S_{u_1,\dots,u_k}$, we conclude that
\begin{equation}
\label{INEQ:split-gcd-total-bound}
S \le \tau_k(z)
\cdot O(d^{\omega(z)})^k \lvert \Delta_Q \rvert^{k/2}
\cdot \left(N^{k-2} + \frac{N^{k-1}}{\lambda^{1/e_P}} + \frac{N^k}{z^{1/e_P}}\right).
\end{equation}
Since $T\le I_1+\dots+I_k + S$, the desired result follows from \eqref{INEQ:isolated-gcd-bound} and \eqref{INEQ:split-gcd-total-bound}.
\end{proof}

To further estimate the upper bound in \eqref{eqn: 2.6} for the quantity \eqref{quantity-T},
we show that on average over certain values of $z$, the quantity $G_{P,\lambda}([N], z)$ appearing in \eqref{eqn: 2.6} is not large.
We show this in Lemma~\ref{LEM:BP-bound}, with Lemma~\ref{LEM:no-linear-factors} and Lemma~\ref{thm: BP} as preparations. 

Lemmas~\ref{LEM:no-linear-factors} and~\ref{thm: BP} imply that when $a\neq b$,
the polynomial $a\cdot P(y) - b \cdot P(x)$ has no linear factor, and therefore has few integral zeros in boxes.
Lemma~\ref{thm: BP} is a well-known counting result due to Bombieri and Pila \cite[Theorem~4]{BP}, which is also used in \cite{KSX}. 

\begin{lemma}
\label{LEM:no-linear-factors}
Let $P, d$ be as above.
Let $a,b$ be distinct positive integers.
Then the polynomial $a\cdot P(y) - b\cdot P(x)$ has no factor in $\C[x,y]$ of degree $1$.
\end{lemma}

\begin{proof}
% \todo[inline]{Check/prove this?\\
% Initial thought: View $aP(y)-bP(x)$ as a polynomial in the ring ${R}[ y]$ where $R=\Z[x]$ is a UFD. So $aP(y)-bP(x)$ has no linear factor is equivalent to 
% it has no root in $R$, i.e. there is no $y=L(x)\in \Z[x]$ such that $aP(L(x))-bP(x) = 0$. Suppose  there is such a 
% $L(x)$. By comparing the degree, $L(x) =ex+f$ for some integer $e, f$  with $e\neq 0$, i.e
% \[ a P(ex+f) = bP(x). \]
% We aim to show this can not hold if $a\neq b$ (and $P(x)$ is irreducible). 
% Suppose $P(x)$ is monic, one must have
% \[ae^{d} =b.\]
% As $a\neq b$, we must have $e\neq 1$. Then one might be able to find contradiction by comparing one more coefficient. 
% (Maybe differentiate $P$ and induct on degree.
% Or look at the root sets, which are permuted by $ex+f$, but the map $z\mapsto ez+f$ shouldn't have many nontrivial orbits.)-----yes, but it seems we need to exclude the case $P(x)=(x-r)^{d}$.
% }
Suppose for contradiction that $a\cdot P(y) - b\cdot P(x)$ has a linear factor, say $fx+gy+h$ (with $f,g,h \in \C$ and $(f,g) \neq (0,0)$).
By symmetry, we may assume $g\neq 0$.
By scaling, we may then assume $g = -1$.
The factor theorem in $(\C[x])[y]$ now implies
\[ a P(fx+h) = b P(x). \]
Thus $f^d = b/a$ (so $|f| = (b/a)^{1/d} \notin \{0,1\}$),
and the (invertible) affine map $L(x) := fx+h$ on $\C$ permutes the roots of $P$.
Thus each root of $P$ is fixed by $L^{d!}$.
But since $|f| \neq 1$, the map $L^{d!} = f^{d!}x + O_{d,f,h}(1)$ has at most one fixed point.
Thus $P$ has at most one root, contradicting our original requirement that $P$ have at least two distinct roots.
\end{proof}

\begin{lemma}[Bombieri--Pila {\cite[Theorem~4]{BP}}]\label{thm: BP}
Let $\mathcal{C}$ be an absolutely irreducible curve (over the rationals) of degree $r\ge 2$.
If $N\ge \exp(r^{6})$,
then the number of  integral points on $\mathcal{C}$ and inside the square $[0, N] \times [0, N]$ does not exceed $N^{1/r} \exp( 12 \sqrt{r \log N \log \log N})$.
\end{lemma}

Recall the definition of $G_{P,\lambda}([N],z)$ from \eqref{EQN:define-lambda-large-gcd-count}.
The following lemma shows that $G_{P, \lambda}([N], P(y))$ is not large when we average over $y$.

\begin{lemma}
\label{LEM:BP-bound}
Let $P, d$ be as above.
Let $N, \lambda\ge 1$ be integers.
Then
\begin{equation*}
\sum_{y\in [N]} G_{P,\lambda}([N], P(y))
\ll_{d, \ee} \lambda^2\cdot N^{1/2 + \ee}.
\end{equation*}
\end{lemma}

\begin{proof}
Plugging in \eqref{EQN:define-lambda-large-gcd-count}, we see that the lemma concerns the number of solutions $(x,y) \in [N]^2$ to $a\cdot P(y) = b\cdot P(x)$ in total over positive integers $a,b\le \lambda$ with $a<b$.
Consider any such integers $a,b$.
By Lemma~\ref{LEM:no-linear-factors},
each irreducible complex component of the curve $a\cdot P(y) - b\cdot P(x) = 0$ has degree $\ge 2$, and thus by Lemma~\ref{thm: BP}, at most $O_{d, \ee}(N^{1/2 + \ee})$ integral points $(x,y) \in [N]^2$.
The lemma follows upon summing over $a,b$.
\end{proof}

With Corollary~\ref{COR:unlikely-P-divisor} and Lemma~\ref{LEM:BP-bound}, we are ready to prove Theorem~\ref{THM:polynomial-paucity}.
The proof proceeds by using a ``symmetry, congruence condition, and induction'' strategy.
\begin{proof}
[Proof of Theorem~\ref{THM:polynomial-paucity}]
% Recall that
We call a solution to \eqref{eqn: S} \emph{trivial} if $(x_1,\dots,x_k)$ equals a permutation of $(y_1,\dots,y_k)$.
% We assume $P$ has positive leading coefficient.
Also, as explained near the beginning of \S\ref{SEC:proof-of-main-result},
we assume $P(n)>0$ for $n>0$.
Therefore, the contribution to $A_{P,2k}([N])$ from trivial solutions is
\begin{equation}
\label{INEQ:trivial-P-solutions-estimate}
k! N^k + O_k(N^{k-1}),
\end{equation}
since it is at most $k! N^k$ and at least $k! N(N-1)\cdots (N-k+1)$.
It remains to bound $N_{P,2k}([N])$, the number of nontrivial solutions $(x_1, \dots, x_k, y_1, \dots, y_k)\in [N]^{2k}$ to \eqref{eqn: S}.

Let $R_{P, 2k}([N])$ be the contribution to $N_{P, 2k}([N])$ from tuples $(x_1, \dots, x_k, y_1, \dots, y_k)$ satisfying
\begin{equation*}
y_k = \max(y_1,\dots,y_k) = \max(x_1,\dots,x_k) = x_k.
\end{equation*}
Let $N'_{P, 2k}([N])$ be the contribution to $N_{P, 2k}([N])$ from tuples $(x_1, \dots, x_k, y_1, \dots, y_k)$ with
\begin{equation}
\label{COND:y_k-maximal-variable}
y_k = \max(y_1,\dots,y_k) > \max(x_1,\dots,x_k).
\end{equation}
Then by symmetry,
\begin{equation}
\label{EQN:recursive-P-estimate}
N_{P, 2k}([N])
\le k^2\cdot R_{P, 2k}([N])
+ 2k\cdot N'_{P, 2k}([N]).
% \le k^2N\cdot N_{P, 2k-2}([N])
% + 2k\cdot N'_{P, 2k}([N]).
\end{equation}
Since $R_{P, 2k}([N]) \le N\cdot N_{P, 2k-2}([N])$ and $R_{P, 2}([N]) = 0$, it follows from \eqref{EQN:recursive-P-estimate} that
\begin{equation}
\label{EQN:max-j-P-estimate}
N_{P, 2k}([N])
\le k \cdot \max_{1\le j\le k}((k^2 N)^{k-j} \cdot 2j\cdot N'_{P, 2j}([N])).
\end{equation}

We now bound $N'_{P, 2k}([N])$.
Let $M(P)$ be a positive integer (depending only on $P$) such that for all integers $n\ge M(P)$, we have
\begin{equation}
\label{COND:n-large-condition-for-M(P)}
P(n) > \max(P(0),P(1),\dots,P(n-1))
\qquad\text{and}\qquad
P(n)\ge n^d/2.
\end{equation}
Let $M\ge M(P)$ and $\lambda\ge 1$ denote integers to be chosen later.
Then by Corollary~\ref{COR:unlikely-P-divisor} (applied with $z := P(y_k)$ for $y_k\ge M$, using the conditions \eqref{COND:y_k-maximal-variable} and \eqref{COND:n-large-condition-for-M(P)}) and the divisor bound, the quantity $N'_{P, 2k}([N])$ is at most
\begin{equation}
M^{2k}
+ O_{k, d, \Delta_Q, \ee}(N^\ee)
\cdot \sum_{M\le y\le N} \left(N^{k-1} G_{P,\lambda}([N], P(y))
+ \frac{N^k}{y^{d/e_P}}+\frac{N^{k-1}}{\lambda^{1/e_P}}+N^{k-2}\right).
\end{equation}
By Lemma~\ref{LEM:BP-bound}, we conclude that
\begin{equation}
N'_{P, 2k}([N])
\le M^{2k}
+ O_{k, d, \Delta_Q, \ee}(N^\ee)
\cdot (\lambda^2 N^{k-1/2}
+ N^kM^{1-d/e_P}+N^k\lambda^{-1/e_P}+N^{k-1}).
\end{equation}
Choosing $M := \floor{M(P)\cdot N^{1/4}}$ (say) and $\lambda := \floor{N^{1/6}}$, we get (since $1-d/e_P\le -1/e_P$)
\begin{equation}
\label{EQN:final-max-P-divisor-bound}
N'_{P, 2k}([N])
\le O_{k, d, \Delta_Q, M(P), \ee}(N^\ee)
\cdot (N^{k/2} + N^{k-1/4e_P} + N^{k-1/6e_P})
\ll_{k, P, \ee} N^{k-1/6e_P+\ee}.
\end{equation}

Finally, plug \eqref{EQN:final-max-P-divisor-bound} into \eqref{EQN:max-j-P-estimate} to get $N_{P, 2k}([N]) \ll_{k, P, \ee} N^{k-1/6e_P+\ee}$.
This, plus the estimate \eqref{INEQ:trivial-P-solutions-estimate} for trivial solutions, establishes the theorem.
\end{proof}

	\bibliographystyle{abbrv}
	\bibliography{CLT}{}

\begin{thebibliography}{10}

\bibitem{batyrev1995manin}
V.~V. Batyrev and Y.~Tschinkel.
\newblock Manin's conjecture for toric varieties.
\newblock {\em J. Algebraic Geom.}, 7(1):15--53, 1998.

\bibitem{BP}
E.~Bombieri and J.~Pila.
\newblock The number of integral points on arcs and ovals.
\newblock {\em Duke Math. J.}, 59(2):337--357, 1989.

\bibitem{Bourgain2014}
J.~Bourgain, M.~Z. Garaev, S.~V. Konyagin, and I.~E. Shparlinski.
\newblock Multiplicative congruences with variables from short intervals.
\newblock {\em J. Anal. Math.}, 124:117--147, 2014.

\bibitem{CS}
S.~Chatterjee and K.~Soundararajan.
\newblock Random multiplicative functions in short intervals.
\newblock {\em Int. Math. Res. Not. IMRN}, 3:479--492, 2012.

\bibitem{breteche2001toric}
R.~de~la Bret\`eche.
\newblock Compter des points d'une vari\'{e}t\'{e} torique.
\newblock {\em J. Number Theory}, 87(2):315--331, 2001.

\bibitem{Breteche}
R.~de~la Bret\`eche.
\newblock Estimation de sommes multiples de fonctions arithm\'{e}tiques.
\newblock {\em Compositio Math.}, 128(3):261--298, 2001.

\bibitem{GS2001}
A.~Granville and K.~Soundararajan.
\newblock Large character sums.
\newblock {\em J. Amer. Math. Soc.}, 14(2):365--397, 2001.

\bibitem{Harper}
A.~J. Harper.
\newblock On the limit distributions of some sums of a random multiplicative
  function.
\newblock {\em J. Reine Angew. Math.}, 678:95--124, 2013.

\bibitem{Harperhigh}
A.~J. Harper.
\newblock Moments of random multiplicative functions, {II}: {H}igh moments.
\newblock {\em Algebra Number Theory}, 13(10):2277--2321, 2019.

\bibitem{Harper-moving}
A.~J. Harper.
\newblock A note on character sums over short moving intervals, 2022.
\newblock arXiv:2203.09448.

\bibitem{HarperHelson}
A.~J. Harper, A.~Nikeghbali, and M.~Radziwill.
\newblock A note on {H}elson's conjecture on moments of random multiplicative
  functions.
\newblock In {\em Analytic number theory}, pages 145--169. Springer, Cham,
  2015.

\bibitem{HeapLind}
W.~Heap and S.~Lindqvist.
\newblock Moments of random multiplicative functions and truncated
  characteristic polynomials.
\newblock {\em Q. J. Math.}, 67(4):683--714, 2016.

\bibitem{HEAP2022}
W.~Heap, A.~Sahay, and T.~D. Wooley.
\newblock A paucity problem associated with a shifted integer analogue of the
  divisor function.
\newblock {\em J. Number Theory}, 242:660--668, 2023.

\bibitem{Hough}
B.~Hough.
\newblock Summation of a random multiplicative function on numbers having few
  prime factors.
\newblock {\em Math. Proc. Cambridge Philos. Soc.}, 150(2):193--214, 2011.

\bibitem{HuxleyCongruence}
M.~N. Huxley.
\newblock A note on polynomial congruences.
\newblock In {\em Recent progress in analytic number theory, {V}ol. 1
  ({D}urham, 1979)}, pages 193--196. Academic Press, London-New York, 1981.

\bibitem{KSX}
O.~Klurman, I.~D. Shkredov, and M.~W. Xu.
\newblock On the random {C}howla conjecture.
\newblock {\em Geom. Funct. Anal.}, 33(3):749--777, 2023.

\bibitem{KS94}
S.~V. Konyagin and T.~Steger.
\newblock Polynomial congruences.
\newblock {\em Mat. Zametki}, 55(6):73--79, 158, 1994.

\bibitem{Naj}
J.~Najnudel.
\newblock On consecutive values of random completely multiplicative functions.
\newblock {\em Electron. J. Probab.}, 25:Paper No. 59, 28, 2020.

\bibitem{WX2022}
M.~Pandey, V.~Y. Wang, and M.~W. Xu.
\newblock Partial sums of typical multiplicative functions over short moving
  intervals.
\newblock {\em Algebra Number Theory}, 18(2):389--408, 2024.

\bibitem{SoundXu}
K.~Soundararajan and M.~W. Xu.
\newblock Central limit theorems for random multiplicative functions.
\newblock {\em J. Anal. Math.}, 151(1):343--374, 2023.

\bibitem{StewartCongruence}
C.~L. Stewart.
\newblock On the number of solutions of polynomial congruences and {T}hue
  equations.
\newblock {\em J. Amer. Math. Soc.}, 4(4):793--835, 1991.

\end{thebibliography}
\end{document}